\documentclass{amsart}
\usepackage{a4wide}
\usepackage{amsmath}
\usepackage{amsfonts}
\usepackage{latexsym, amssymb}
\usepackage{euscript}
\usepackage[latin1]{inputenc}

\newtheorem{theorem}{Theorem}
\newtheorem{lemma}[theorem]{Lemma}
\newtheorem{corollary}[theorem]{Corollary}
\newtheorem{proposition}[theorem]{Proposition}

\newtheorem{theirtheorem}{Theorem}

\newtheorem{theirlemma}[theirtheorem]{Lemma}

\newcommand{\subgp}[1]{\langle{#1}\rangle}

\def\Z{\mathbb Z}

\def\zn{{\mathbb Z}/n{\mathbb Z}}

\begin{document}

\title{Large restricted sumsets in general Abelian groups}

\author{Y. ould Hamidoune \and S. C. L\'opez \and A. Plagne}
\email{susana@ma4.upc.edu}
\email{plagne@math.polytechnique.fr}
\address{Universitat Polit\`ecnica de Catalunya. BarcelonaTech, Dept. Matem\`atica Apl. IV;
C/Esteve Terrades, 5, Castelldefels, Spain.}
\address{\'Ecole polytechnique, Centre de Math\'ematiques Laurent Schwartz,
UMR 7640 du CNRS, 91128 Pa\-lai\-seau cedex, France.}

\maketitle

\date{\today}

\vspace{.5cm}
\begin{center}
\noindent{\em Soon after this work was started,\\
the first-named author and main inspirator of this article passed
away unexpectedly.\\ The two other authors dedicate the paper to
his memory.}
\end{center}
\vspace{.5cm}

\begin{abstract}
Let $A$, $B$ and $S$ be three subsets of a finite Abelian group $G$.
The restricted sumset of $A$ and $B$ with respect to $S$ is defined
as $A\wedge^{S} B= \{ a+b: a\in A, b\in B\ \mbox{and}\ a-b\notin S\}$.
Let
$L_S=\max_{z\in G}|\{(x,y):\  x,y\in G,\ x+y=z \ \mbox{and}\  x-y\in S\}|$.
A simple application of the pigeonhole principle shows that $|A|+|B|>|G|+L_S$ implies $A\wedge^S B=G$. We then prove
that if $|A|+|B|=|G|+L_S$ then $|A\wedge^S B|\ge |G|-2|S|$. We also
characterize the triples of sets $(A,B,S)$ such that $|A|+|B|=|G|+L_S$
and $|A\wedge^S B|= |G|-2|S|$. Moreover, in this case, we also provide
the structure of the set $G\setminus (A\wedge^S B)$.
\end{abstract}

\section{Introduction}
In this paper $G$ be a finite Abelian group. Given two subsets $A$ and $B$ of $G$,
the {\em sumset} and the {\em restricted sumset} of $A$ and $B$ are defined,
respectively, by
$$
A+B=\{a+b: a\in A, b\in B\}\ \text{ and }\
A \wedge B = \{a+b:a\in A, b\in B\ \mbox{and}\ a\neq b\}.
$$
We shall write $A+b$ or $A-b$ instead of $A+\{ b \}$ and $A + \{ -b \}$.

To give lower bounds for the cardinality of sumsets is probably the most
central problem of additive number theory (see \cite{Nat96} for a general
overview). A historical result in this area is the famous Cauchy-Davenport
theorem \cite{Cauchy, Davenport}.

\begin{theirtheorem}[Cauchy, Davenport]
Let $A$ and $B$ be non-empty subsets of the group of prime order $p$.
Then
$$
|A+B|\ge\min\{p,|A|+|B|-1\}.
$$
\end{theirtheorem}

For restricted sumsets, the most famous result is due to Dias da Silva and
Hamidoune \cite{DiaHam94} who, in the beginning of the 1990s, solved
an Erd\H{o}s-Heilbronn conjecture which remained open since 1964:

\begin{theirtheorem}[Dias da Silva and Hamidoune]
Let $A$ and $B$ be non-empty sets of the group of prime order $p$.
Then
$$|A\wedge B|\ge \min\{p,|A|+|B|-3\}.$$
\end{theirtheorem}

Several years later, Alon, Nathanson and Rusza \cite{AloNatRus96} proposed
an alternative proof using the so-called polynomial method, a powerful method
which has then inspired a lot of new results in additive combinatorics.

Another important set of problems in the area is known under the name of
{\em Critical Pair Theory}. Having found a general lower bound for the
cardinality of sumsets, the problem is now to achieve the description
of pairs of sets, the sum of which attains the lower bound.
For instance, Vosper's Theorem describe precisely the pairs of subsets
$(A,B)$ in a group of prime order such that $|A+B|=|A|+|B|-1$.

We now introduce some notation and tools needed in the sequel and formulate
a few basic remarks.

Let $G_0$ denote the subgroup of $G$ composed of elements of order 2
or less, that is,
$$
G_0=\{x\in G:\ 2x=0\}.
$$
We write $L(G)= |G_0|$, the {\em doubling constant} introduced by Lev
in \cite{Lev00}. Notice that if $|G|$ is odd, then $L(G)=1$ whereas
if $|G|$ is even, $L(G)$ is a power of $2$.
It is immediate to notice that
$L(G)$ is
the maximal number of pairwise distinct elements of a group $G$
that can share a common doubling, in other words
$$
L(G) = \max_{t \in G} |\{x\in G:\ 2x=t\}|.
$$

If two sets $A$ and $B$ of $G$ are given, we denote for any $x \in G$,
$$
\nu (x)=|\{(a,b): \ a\in A, \ b\in B,\ a+b=x\}|,
$$
the {\em number of representations function}. When the context will
not make it obvious, we may denote $\nu_{A,B}$ instead of simply $\nu$.
Notice that if $|A|+|B|>|G|+L(G)$, then for any $x\in G$ we have
$$\nu (x)=|A\cap (x-B)|=|A|+|x-B|-|A\cup (x-B)|\ge |A|+|B|-|G|>L(G),$$
whence
$A\wedge B=G$.

Finally, for a set $A \subset G$ and $t\in \Z$, we shall denote
$$
t\cdot A =\{ ta: a \in A \},
$$
and $-A= (-1) \cdot A$. We define the {\em half} of a set $T\subset G$ as  $\mathcal{H}(T)=\{g\in G: \ 2g\in T\}$ and the {\em subgroup
of doubles} in $G$ as
$$
2\cdot G = \{ 2x :\ x \in G \}.
$$
Notice that $|2\cdot G|=|G|/L(G)$.

In a recent paper, Guo \cite{Guo09} studied the problem of
restricted sumsets in Abelian groups in the case when
the cardinality of the sets is large and proved the following result.

\begin{theirtheorem}[Guo]
\label{Guo_theo_bound}
Let $A$ and $B$ be subsets of a finite Abelian group $G$
satisfying $|A|+|B|=|G|+ L(G)$. Then $|A\wedge B|\ge |G|-2$.
\end{theirtheorem}

In the same paper \cite{Guo09}, Guo also gave a complete description of
the pairs of subsets $(A,B)$ such that $|A|+|B|=|G|+L(G)$ and
$|A\wedge B|=|G|-2$. This is the content of the next theorem.

\begin{theirtheorem}[Guo]
\label{Guo_theo_charac}
Let $A$ and $B$ be subsets of a finite Abelian group $G$.
Then, $|A|+|B|=|G|+ L(G)$ and $|A\wedge B|= |G|-2$ if and only if
there exist two distinct elements $a,b\in A\cap B$ satisfying
\begin{itemize}
\item[(i)] the order $d$ of the subgroup $H=\subgp{2(b-a)}$
is an odd integer greater than $1$;
\item[(ii)] there exist distinct elements $x_1,\ldots ,x_k,x_{k+1},\ldots,
x_l,x_{l+1},\ldots, x_m$ in $G \setminus (G_0+H)$, where
$$
m=\frac{|G|/d-|G_0|}{2}
$$
and $0\le k,l\le m$, such that
\begin{eqnarray*}
&&\hspace{-1cm}G\setminus (G_0+H) =\bigcup_{i=1}^m((a+x_i+H)\cup (a-x_i+H)), \\
&&\hspace{-1cm}A=
a+((\{0,b-a,3(b-a),\ldots, (d-2)(b-a)\}+G_0)\cup (\{x_1,\ldots,x_k,\pm x_{k+1},\ldots, \pm x_l\}+H)), \\
&&\hspace{-1cm}B =
a+((\{0,b-a,3(b-a),\ldots, (d-2)(b-a)\}+G_0)\cup (\{x_1,\ldots,x_k,\pm x_{l+1},\ldots, \pm x_m\}+H)).
\end{eqnarray*}
\end{itemize}
\end{theirtheorem}

We want to point out that, as it was proved in \cite{Guo09}, if
two sets $A$ and $B$ are of the form described in Theorem
\ref{Guo_theo_charac}, then the two {\em $A\wedge B$-exceptions} --
by which we mean elements of $(A+B) \setminus (A \wedge B)$ --
are precisely $2a$ and $2b$.

In this article, we deal with a generalization of restricted sumsets
(introduced in \cite{Sun})
in which a new set appears. Let $A,B$ and $S$ be non-empty subsets of $G$.
The {\em restricted sumset of $A$ and $B$ with respect to $S$} is defined
by
$$
A\wedge^S B= \{ a+b:\ a\in A,\ b\in B \ \mbox{and}\  a-b\notin S\}.
$$
Notice that, when $S=\{0\}$, this sumset corresponds to the classical
restricted sumset of two sets. Partial results to the problem of
estimating $|A\wedge^S B|$ from below are given recently in
\cite{GuoSun09,PanSun02,PanSun06}. In particular, Pan and Sun used
the polynomial method to study a conjecture of Lev. As a corollary
they proved \cite{PanSun06} the following result.

\begin{theirtheorem}[Pan and Sun]
\label{PanSun06}
Let $G$ be an Abelian group and let $A,B$ and $S$ be finite non-empty
subsets of $G$ such  that $A\wedge^SB$ is not empty.
\begin{itemize}
\item[(i)] If $G$ is torsion-free or elementary Abelian, then
$$
|A\wedge^SB|\ge |A|+|B|-|S|-\min_{z\in A\wedge^SB}\nu(z).
$$
\item[(ii)] If the torsion part of $G$ is cyclic, then
$$
|A\wedge^SB|\ge |A|+|B|-2|S|-\min_{z\in A\wedge^SB}\nu(z).
$$
\end{itemize}
\end{theirtheorem}

In \cite{GuoSun09}, Guo and Sun used a variation of Tao's method
\cite{Tao05} in harmonic analysis to prove the next theorem, which
is a generalization of Theorem B.

\begin{theirtheorem}[Guo and Sun]
Let $A,B$ and $S$ be non-empty subsets of the group of prime order $p$. Then
$$
|A\wedge^S B|\geq \min\{p,|A|+|B|-2|S|-1\}.
$$
\end{theirtheorem}

In this paper, applying techniques similar to those used  in \cite{Chi02,GalGre} and \cite{Guo09}, we study the restricted sumset of two large sets
$A$ and $B$ with respect to a set $S$, in general finite Abelian
groups. Given non-empty sets $A,B$ and $S$, we first introduce
a generalization of the doubling constant $L(G)$ which depends on
the set $S$, that we denote $L_S$.
It is easy to see that if $|A|+|B|>|G|+L_S$ then  $A\wedge^S B=G$ (Lemma \ref{phS}) and
as our first principal result we show that if $|A|+|B|=|G|+L_S$ then $|A\wedge^S B|\ge |G|-2|S|$  (Theorem \ref{general}).
We also characterize the triples of sets $(A,B,S)$ such that
$|A|+|B|=|G|+L_S$ and $|A\wedge^S B|= |G|-2|S|$
(Theorems \ref{characterization_even} and
\ref{characterization_general}).
Moreover, in this case, we also provide the structure of the set
$G\setminus (A\wedge^S B)$.

The organization the paper is the following. In Section 2,
we introduce the terminology and some preliminary results.
The key-point in this section is Lemma \ref{clau}.
Section 3 is devoted to the proof of the lower bound.
We also give some examples in order to show that our bound is tight.
In Section 4, we characterize the critical sets
in the important special case $L_S=|S|\ L(G)$. In particular, this gives
the characterization the critical sets for Abelian groups of
odd order. Finally, in Section 5, we extend the characterization
to the case $L_S<|S|\ L(G)$, provided some restriction holds.

\section{Terminology and preliminaries}

We first start with two basic results. The first one was
baptised Prehistorical lemma by the first-named author.

\begin{theirlemma}[Folkloric prehistorical lemma]
\label{ph}
Let $A$ and $B$ be subsets of a finite group $G$.
If $|A|+|B|>|G|$ then $A+B=G$.
\end{theirlemma}

The second result we shall need is Kneser's Theorem \cite{Kneser},
see also \cite{Nat96}.
It has a lot of applications in additive and combinatorial number theory.
We will use it as a key-tool in the characterization of critical sets
in groups of even order.

\begin{theirtheorem}[Kneser]
\label{kneser}
Let $G$ be an Abelian group and let $A$ and $B$ be finite,
non-empty subsets of $G$.
Let $H=H(A+B)=\{g\in G: \ g+A+B=A+B\}$ be the stabilizer of $A+B$.
If $|A+B|<|A|+|B|$ then $|A+B|=|A+H|+|B+H|-|H|$.
\end{theirtheorem}

Given $z\in G$ and a non-empty subset $S\subset G$, we define
$$
L_S(z)=|\{(x,y):\  x,y\in G,\ x+y=z \ \mbox{and}\  x-y\in S\}|
$$
and
$$
L_S = \max_{z \in G} L_S (z).
$$
The mean-value of $L_S(z)$ on $G$ is easy to compute since
\begin{eqnarray*}
\frac{1}{|G|} \sum_{z \in G} L_S(z) & = &
\frac{1}{|G|} \sum_{z \in G} |\{(x,y):\  x,y\in G,\ x+y=z \ \mbox{and}\  x-y\in S\}|\\
		&=&  \frac{1}{|G|}  |\{(x,y):\  x,y\in G,\  \mbox{such that }\  x-y\in S\}|\\
& =&  \frac{1}{|G|}\ |S| |G| = |S|.
\end{eqnarray*}
Therefore, we must have $L_S \geq |S| \geq 1$.

The next lemma will be useful
for further reference. Notice that $L_S(z)=|\{y: \ z-2y\in S\}|.$

\begin{lemma}
\label{L_S}
Let $S$ be a finite non-empty subset of an Abelian group $G$ and let
$z\in G$. Then
\begin{itemize}
\item[(i)]  We have
$$
L_S(z)=
|S \cap (z+2\cdot G) |\ L(G).
$$
\item[(ii)]
In particular, $L_S(z)$ is a multiple of $L(G)$, less than or equal to
$|S|L(G)$.
\item[(iii)] In particular, $L_S(0)=|S\cap 2\cdot G|\ L(G)$.
\item[(iv)]
In particular,
$$
L_S= m\ L(G),
$$
for some integer $m$ satisfying $1 \leq m \leq |S|$.
\item[(v)] If $L_S= |S|\ L(G)$, then $S$ is included in a coset
modulo $2 \cdot G$.
\item[(vi)] If $L_S(0)= L_S= |S|\ L(G)$, then $S$ is included in $2 \cdot G$.
\end{itemize}
\end{lemma}

The next result can be thought as a generalization of 
 Lemma \ref{ph}:
\begin{lemma}
\label{phS}
Let $A,B$ and $S$ be non-empty subsets of a finite Abelian group $G$.
If $|A|+|B|> |G|+L_S$ then $A\wedge^S B=G$.
\end{lemma}

\begin{proof}
For any $z\in G$ we have
$$\nu (z)=|A\cap (z-B)|=|A|+|z-B|-|A\cup (z-B)|\ge |A|+|B|-|G|>L_S.$$
Thus, by definition of $L_S$, one at least among
the $\nu(z)$ pairs $(a,b) \in A \times B$ such that $z=a+b$ must satisfy
$a-b\notin S$. Therefore $z \in A\wedge^S B$.
\end{proof}

Similarly, we obtain the following:

\begin{lemma}
\label{exception}
Let $A,B$ and $S$ be non-empty subsets of a finite Abelian group $G$ and
assume that $|A|+|B|= |G|+L_S,$ in particular $A+B=G$. Then
\begin{itemize}
\item[(i)] For each $z \in G$, $\nu (z)\ge L_S$.
\item[(ii)] If $z \notin A\wedge^SB $ then $\nu(z)=L_S(z)=L_S$.
That is, there are exactly $L_S$ pairs $(a,b)\in A\times B$ such that
$z=a+b$. Moreover, for each sum $z=a+b$ with $a\in A$ and $b\in B$ we have
$a-b\in S$.
\end{itemize}
\end{lemma}

\begin{proof}
For any $z\in G$ we have
$\nu (z)=|A\cap (z-B)|=|A|+|z-B|-|A\cup (z-B)|\ge |A|+|B|-|G|=L_S$. This proves (i).

For the proof of (ii), by definition, if $z=a+b$ with $a\in A$ and
$b\in B$ then $a-b\in S$. Thus,
$\nu (z) \le L_S(z) \le L_S$. Hence, by (i), we must have $\nu (z)=L_S(z)=L_S$.
\end{proof}

Generalizing an earlier notation, we say that $z$ is an
{\em $A\wedge^S B$-exception} if $z\in (A+B)\setminus (A\wedge^SB)$.

A useful reduction is given by the following lemma.

\begin{lemma}
\label{xxl}
Let $A,B$ and $S$ be non-empty subsets of a given Abelian group $G$.
Let $z$ be an $A\wedge^S B$-exception.
\begin{itemize}
\item[(i)] There exist $s\in S$ and $b\in (A-s)\cap B$ such that
$z=2b+s$.
\item[(ii)] For any $s\in S$ and $b\in (A-s)\cap B$ with
$z=2b+s$, letting $A' = A-b-s$, $B'=B-b$ and $S'=S-s$,
we have
\begin{itemize}
\item[(a)] $0\in A'\cap B'\cap S'$,
\item[(b)] $A'\wedge^{S'} B' = A \wedge^S B -(2b + s)$, and
\item[(c)] $0$ is an $A'\wedge^{S'} B'$-exception.
\end{itemize}
\item[(iii)] If $|A|+|B|=|G|+L_S$ then
$$
\nu_{A',B'}(0)=L_{S'}(0)=L_{S'}=L_S.
$$
\item[(iv)] If $|A|+|B|=|G|+L_S$ and $L_S= |S|\ L(G)$ both hold,
then (i), (ii) and (iii) are valid for any $s \in S$.
\end{itemize}
\end{lemma}

\begin{proof}
Since $z$ is in $A+B$, it can be written $a+b$ for some $a \in A$ and $b \in B$.
But, being an $A\wedge^S B$-exception, we must have $a-b \in S$, say
$a-b=s$, for an $s \in S$. It follows that $b \in A-s$ and
$z=a+b=2b+s$. This proves (i).

To prove (ii), one first observes that $A'=A-a$ and thus (a) is immediate.
Assertion (b) follows from
\begin{eqnarray*}
A'\wedge^{S'} B' & = &
\{ a'+b': a'\in A',\ b'\in B' \ \mbox{and}\  a'-b'\notin S'\} \\
	&=& \{ (\alpha -a)+(\beta -b) :\ \alpha \in A,\ \beta \in B \
\mbox{and}\  (\alpha -a) - (\beta -b) \notin S-s \} \\
	&=& \{ \alpha +\beta -(a+b):\ \alpha \in A,\ \beta \in B \
\mbox{and}\  \alpha - \beta \notin S \} \\
	&=& A\wedge^{S} B - (a+b)\\
	&= &A\wedge^{S} B - (2b+s).
\end{eqnarray*}
Finally (c) follows from the fact that if $0$ (which belongs to $A'+B'$)
was not an $A'\wedge^{S'} B'$-exception, then $0=a'+b'$ with $a' \in A'$,
$b' \in B'$ and $a'-b' \not\in S'$ from which we would derive that $z$
is not an $A\wedge^{S} B$-exception, a contradiction.

Point (iii) is an immediate consequence of Lemma \ref{exception} (ii)
applied to the sets $A',B'$ and $S'$ and the $A'\wedge^{S'}
B'$-exception $z=0$, on recalling that $|A'|=|A|$,
$|B'|=|B|$, $|S'|=|S|$ and $L_{S'}=L_S$.

For the proof of (iv), let us write $k=|S|$, $S=\{s_1,\dots,s_k\}$
and consider $z$ an $A\wedge^S B$-exception.
By Lemma \ref{exception} (ii) again, we know that there are exactly $L_S$ 
pairs $(a,b)\in A\times B$ such that $z=a+b$ and, for each such sum,
we have $a-b\in S$ or, equivalently $a-b=s_i$ for some $i$ in $\{
1,\dots, k\}$. But for each $1 \leq i \leq k$, there are at
most $L(G)$ solutions to the system $a+b=z,\ a-b=s_i$, since it is
equivalent to the equation $z=2b + s_i$.
Since $L_S = |S|\ L(G)$, the
only possibility is that for all $i$  in $\{ 1,\dots, k\}$, there are
exactly $L(G)$ corresponding solutions.
In particular, this implies that there is at least
one solution $(a,b) \in A \times B$ to $a+b=z$ and $a-b=s_i$, for each $i \in \{1,\dots,k \}$.
\end{proof}

\begin{corollary}
\label{coroxxl}
Let $A,B$ and $S$ be non-empty subsets of a given Abelian group
$G$. We assume that $0 \in S$ and that both equalities
$|A|+|B|=|G|+L_S$ and $L_S= |S|\ L(G)$ hold, then
$$
G\setminus (A\wedge^{S} B) \subset 2 \cdot (A \cap B).
$$
\end{corollary}

\begin{proof}
This follows from (iv) in the preceding lemma: one may therefore apply
(i) with any prechosen element $s$ in $S$. Selecting $s=0$, it follows
that any $A\wedge^{S}B$-exception $z$ can be written
$2b$ for some $b\in  A \cap B$.
\end{proof}

The next lemma is a technical result which will be central to give
a lower bound for $|A\wedge^SB|$.

We shall use the standard notation
$$
S\Delta -S = (S \setminus -S)\cup (-S \setminus S)
$$
for any set $S$ of $G$. 
 Notice that
$$
(S\cap -S)\cup (S\Delta -S)=(S\cap -S)\cup (S \setminus -S)\cup (-S \setminus S)
$$
is a partition of $S\cup -S$.

\begin{lemma}
\label{clau}
Suppose that $A,B$ and $S$ are subsets of a finite Abelian group such that $|A|+|B|= |G|+L_S$, $L_S=|S|\ L(G)$ and $0\in  A\cap B\cap S$. If $b\in A\cap B$ and $0,2b\notin A\wedge^S B$, then for any $x\in G$ we have
$$
|A\cap \{b-x,b+x\}|+|B\cap \{b-x,b+x\}|=
\left\{
\begin{array}{cl}
2 & \text{if } 2x \notin S\cup -S, \\
3 & \text{if } 2x \in S\Delta -S, \\
4 & \text{if } 2x \in S\cap -S,\ 2x\neq 0 \\
2 & \text{if } 2x =0 \\
\end{array}
\right.
$$

\end{lemma}

\begin{proof}
We shall denote by LHS$(x)$ and RHS$(x)$, respectively, the left-hand side and the right-hand side of the equality to prove. From $2b\notin  A \wedge^S B$ we easily get  $\hbox{LHS}(x)\le \hbox{RHS}(x)$, for any $x\in G$. On  the other hand, if $\sum_x^*$ denotes the summation over all elements $x\in G$, with every value $x$ with $2x=0$ attained twice, then
$${\sum_x}^*\hbox{LHS}(x)=2|A|+2|B|.$$

Furthermore, by Lemma \ref{L_S} (vi) we have $S\cup (-S)\subset 2\cdot G$, implying:
\begin{eqnarray*}
  {\sum_x}^*\hbox{RHS}(x) &=& 2|G|+|\{x\in G:\ 2x\in S\Delta -S\}|+2|\{x\in G:\ 2x\in S\cap -S\}|\\
  &=& 2|G|+|\{x\in G:\ 2x\in S\cup -S\}|+|\{x\in G:\ 2x\in S\cap -S\}|\\
   &=& 2|G|+|S\cup -S|L(G)+|S\cap -S|L(G) \\
   &=& 2|G|+2|S|L(G) \\
  &=& 2|G|+2L_S \\
  &=&2|A|+2|B|.
\end{eqnarray*}
\end{proof}

From Lemma \ref{clau} we derive the following corollary.

\begin{corollary}
Suppose that $A, B$ and $S$ are subsets of a finite Abelian group $G$ such that $|A|+|B|= |G|+L_S$, $L_S=|S|\ L(G)$, and $0\in A\cap B\cap S$. If $b\in A\cap B$ and
$0,2b\notin A\wedge^S B$, then there exist partitions
$$G\setminus \mathcal{H}(S\cup -S)=X_0\cup -X_0\cup X_1\cup -X_1\cup X_2\cup -X_2$$
and
$$\mathcal{H}(S\Delta -S)=Y_1\cup -Y_1\cup Y_2\cup-Y_2$$
such that
\begin{eqnarray*}
A-b&=&\mathcal{H}(S\cap -S)\cup X_0\cup X_1\cup-X_1\cup Y_1\cup -Y_1\cup Y_2,\\
B-b&=&\mathcal{H}(S\cap -S)\cup X_0\cup X_2\cup-X_2\cup Y_1\cup Y_2\cup -Y_2.
\end{eqnarray*}
\end{corollary}

\section{The lower bound}

We start this section with a lemma that contains the central part
of the main result.

\begin{lemma}\label{PofTheorem7}
Let $A,B$ and $S$ be subsets, containing $0$,
of a finite Abelian group $G$. Assume that $|A|+|B|= |G|+L_S$,
that $0$ is an $A\wedge^S B$-exception and that $L_S=|S|\ L(G)$.
Let $z$ be an $A\wedge^S B$-exception not contained in $S\cup -S$.
\begin{itemize}
\item[(i)] If $z'$ is another $A\wedge^S B$-exception, then we have $z'-z\in A\cap B$,
\item[(ii)] Moreover, if $z'\notin S\cup -S$ then $z'-z\in S\cap -S$.
\end{itemize}
\end{lemma}

\begin{proof}
By Corollary \ref{coroxxl}, we may assume that $z=2b$ and $z'=2b'$ for some $b,b' \in A\cap B$.

Clearly $2b \notin A\cup B$, otherwise since $0\in A\cap B$
and $2b$ is an $A\wedge^{S}B$-exception, we would have that
either $2b\in S$ or $-2b\in S$ that is, $2b \in S\cup -S$,
a contradiction. Defining $r=2b-b'$, this can be reformulated as
$b'+r \notin A\cup B$. Therefore, since $2b'$ is an $A\wedge^{S}B$-exception,
by Lemma \ref{clau},
we must have $b'-r \in A\cap B$, or equivalently $2b'-2b \in A\cap B$.
This proves (i).

Suppose now that $2b'\notin S\cup -S$. By symmetry, applying what
we just proved, we also have $2b-2b'\in A \cap B$.

But then, using the fact that $0$ is an $A\wedge^{S}B$-exception,
the equalities (giving two representations of $0$ as an element of $A+B$)
$$
(2b-2b')+(2b'-2b)=(2b'-2b)+(2b-2b')=0
$$
imply that both $4b-4b'$ and $4b'-4b$ are in $S$, that is $2(2b'-2b)\in S\cap -S$.

Now, Lemma \ref{clau} implies that $2b'-b= b+(2b'-2b)\in  A\cap B$.
Using this and  $b \in B$, we deduce from
$2b'=(2b'-b)+b\in A+B$ and the fact that  $2b'$ is an $A\wedge^{S}B$-exception,
that $2b'-2b= (2b'-b)-b \in S$. By symmetry, $2b-2b'\in S$, therefore
$2b'-2b \in S\cap -S$, which proves assertion (ii).
\end{proof}

\begin{corollary}
\label{CoroPofThm7}
Let $A,B$ and $S$ be subsets, containing $0$,
of a finite Abelian group $G$. Assume that $|A|+|B|= |G|+L_S$,
that $0$ is an $A\wedge^S B$-exception and that $L_S=|S|\ L(G)$.
Then, for any $A\wedge^S B$-exception $z$
not contained in $S\cup -S$, we have
$$
G\setminus (A\wedge^S B)\subset S\cup -S \cup (z+S\cap -S ).
$$
In the special case $S=-S$, we obtain
$G \setminus (A\wedge^S B)\subset S \cup (z+S )$.
\end{corollary}

\begin{proof}
Since, by assumption, $z$ is not in $S\cup -S$, we may then apply the
preceding lemma. Consider another $A\wedge^S B$-exception $z'$.
By Lemma \ref{PofTheorem7} (ii), it is either in $S\cup -S$,
or in $z+S\cap -S$.
\end{proof}

The next theorem gives us the lower bound for restricted sumsets
with respect to a set $S$, in the case $L_S=|S|\ L(G)$.
It is a direct application of Lemma \ref{PofTheorem7}.

\begin{theorem}
\label{lower_2s}
Let $A,B$ and $S$ be non-empty subsets of a finite Abelian
group $G$. If $|A|+|B|=|G|+L_S$ and $L_S=|S|\ L(G)$, then
$$
|A\wedge^S B|\ge |G|-2|S|.
$$
\end{theorem}

\begin{proof}
Since $L_S\ge 1$, Lemma \ref{ph} implies that $A+B=G$. Assume that
the set of exceptions $G\setminus (A\wedge^S B)$ is not empty, otherwise
there is nothing to prove. Let $z$ be an $A\wedge^S B$-exception.

Lemma \ref{xxl} gives the existence of $s \in S$ and
$b\in (A-s)\cap B$ such that $z=2b+s$ and if we put
$A'=A-b-s$, $B'=B-b$ and $S'=S-s$, we have that $|A'|+|B'|=|G|+L_{S}$,
$0\in A'\cap B'\cap S'$, $0$ is an $A'\wedge^{S'} B'$-exception and
that $L_{S}=L_{S'}=L_{S'}(0)$.

If there is no $A'\wedge^{S'}B'$-exception outside $S'\cup -S'$, then
$$
|G\setminus (A\wedge^{S}B)| = |G\setminus (A'\wedge^{S'}B')|
\le |S'\cup -S'| \le 2|S'| = 2 |S|
$$
and thus the result holds.

Suppose now that there is at least one $A'\wedge^{S'}B'$-exception
outside $S'\cup -S'$. By applying Corollary \ref{CoroPofThm7}, we obtain
\begin{eqnarray*}
|G \setminus (A\wedge^{S}B) | & = & |G \setminus (A' \wedge^{S'} B')| \\
	& \leq & |S'\cup -S'|+ |S'\cap -S'| \\
	& = & 2|S'| \\
	& = & 2|S|
\end{eqnarray*}
which proves the theorem.
\end{proof}

We introduce another lemma.

\begin{lemma}
\label{texnik}
Let $A,B$ and $S$ be subsets, containing $0$,
of a finite Abelian group $G$. Assume that  $|A|+|B|=|G|+L_S$ and
that $0$ is an $A\wedge^{S} B$-exception.
Let $\Sigma = S \cap 2 \cdot G$. Then
\begin{itemize}
\item[(i)] $L_S = |\Sigma|\ L(G)=L_{\Sigma}(0) =L_{\Sigma}$,
\item[(ii)] $|A\wedge^{\Sigma} B|\ge |G|-2|\Sigma|$, and
\item[(iii)] the set of
$A\wedge^{S} B$-exceptions can be partitioned as follows:
$$
G\setminus (A\wedge^{S} B)= (G \setminus (A\wedge^{\Sigma} B))
\cup ( G\setminus (A\wedge^{(S\setminus \Sigma)} B)).
$$
\end{itemize}
\end{lemma}

\begin{proof}
By Lemmas \ref{L_S} and \ref{exception}, since $0$ is an $A\wedge^{S} B$-exception, we have that
$$
L_S=|\Sigma|\ L(G).
$$
By applying Lemma \ref{L_S} (iii), we check that
$L_{\Sigma}(0)=|\Sigma|\ L(G)$ and by definition of $L_{\Sigma}$
and Lemma \ref{L_S} (iv), we have
$L_{\Sigma}(0) \leq L_{\Sigma} \leq | \Sigma |\ L(G)$, therefore
$$
|\Sigma|\ L(G)=L_{\Sigma}(0) =L_{\Sigma}
$$
and (i) is proved.

In view of (i), Theorem \ref{lower_2s} applied to the sets $A,B$ and
$\Sigma$ implies that
\begin{equation}
\label{Sigma}
\nonumber
|A\wedge^{\Sigma} B|\ge |G|-2|\Sigma|,
\end{equation}
that is (ii).

To prove (iii), first notice that it is immediate that the
right-hand side is included in the left-hand side. Let us see
now that the other inclusion holds. Assume that $z$ is an
$A\wedge^{S} B$-exception. By Lemma \ref{xxl} (i),
there are $s\in S$ and $b\in (A-s)\cap B$ such that $z=2b+s.$
Assume that there exists a different $s'\in S$ such that $z=2b'+s'$,
where $b'\in (A-s')\cap B$. Thus, we obtain that
$s=s'+2(b'-b)\in s'+2 \cdot G$. That is, if $s\in 2 \cdot G$
then $s'\in 2 \cdot G$, and viceversa. Hence, $s\in \Sigma$ if and only if,
 $s'\in \Sigma$. Therefore, if $z$ is an
$A\wedge^{S} B$-exception and there exist $a,a'\in A$ and $b,b'\in B$ such that $z=a+b=a'+b'$ either $a-b,a'-b'\in \Sigma$ or $a-b,a'-b'\in S\setminus \Sigma$.
\end{proof}

We are now ready for our next theorem which is
a generalization of Theorem \ref{Guo_theo_bound}.

\begin{theorem}
\label{general}
Let $A,B$ and $S$ be non-empty subsets of a finite Abelian group $G$.
If $|A|+|B|=|G|+L_S$  then  $|A\wedge^S B|\ge |G|-2|S|.$
\end{theorem}

Notice that, for the sake of clarity, (the first step of the induction in)
the forthcoming proof relies on Theorem C, but it would be no problem
-- to be more precise, the very same proof, but in a drastically
simplified way -- to keep the paper self-contained and prove Theorem
\ref{general} without invoking it.

\begin{proof}

We shall prove the result by (finite) induction on the cardinality of $S$.

If $|S|=1$, then $L_S=L(G)$ and the result holds by Theorem C.

Assume now that the result is proved for any set of cardinality
$< \sigma$ for some integer $\sigma \leq |G|$ and let us consider
a set $S$ of cardinality $\sigma$.

We start by recalling that in the present situation, $A+B=G$.
We may also assume that there is at least one $A\wedge^S B$-exception,
say $z$, otherwise $A\wedge^S B=G$ and there is nothing more to prove.

By Lemma \ref{xxl} (i), we can assume that $z=2b+s$ for some $s\in S$ and
$b\in (A-s)\cap B$.
By replacing $A$, $B$ and $S$ respectively by $A'=A-b-s$,
$B'=B-b$ and $S'=S-s$ we have that $0\in A'\cap B'\cap S'$,
$0$ is an $A'\wedge^{S'} B'$-exception and that
$L_S=L_{S'}(0)=L_{S'}$. Moreover we also have $A' + B' =G$.

We denote $\Sigma = S' \cap 2 \cdot G$ and notice that $0\in \Sigma$.
We distinguish three cases.

If $\Sigma=S'$, then by Lemma \ref{texnik} (i), we can apply
Theorem \ref{lower_2s}, which gives the result. From now on,
assume $\Sigma \neq S'$.

If $A'\wedge^{(S'\setminus \Sigma)} B'=G$ then, by Lemma \ref{texnik} (iii), we obtain that
$$
G\setminus (A'\wedge^{S'} B')= G \setminus (A'\wedge^{\Sigma} B').
$$
Lemmas \ref{xxl} and \ref{texnik} (ii) now yield
$$
|A \wedge^S B| = |A' \wedge^{S'} B' | = |A'\wedge^{\Sigma} B'|
\ge |G|-2|\Sigma|\ge |G|-2 |S'|=|G|-2|S|,
$$
and the result is proved.

Or, there exists an exception
$z' \in G\setminus (A'\wedge^{(S'\setminus \Sigma)} B')$.
We first check that
\begin{equation}
\label{lesl}
L_{S' \setminus \Sigma}= L_{S'}.
\end{equation}
Indeed,  by Lemma
\ref{texnik} (iii) $z'$ is an $A'\wedge^{S'} B'$-exception, and by Lemma
\ref{exception} (ii), $\nu (z')=L_{S'}=L_{S'}(z')$. But, since $z' \notin A' \wedge^{S' \setminus \Sigma} B'$,
for any $a \in A'$, $b \in B'$ such that $z'=a+b$, we have
$a-b \in S' \setminus \Sigma$. Thus
$\nu (z)\le L_{S' \setminus \Sigma} (z') \le L_{S' \setminus \Sigma}\le  L_{S'}$ which implies
equality \eqref{lesl}.

By \eqref{lesl}, we have $|A'|+|B'|=|G|+L_{S' \setminus \Sigma}$.
Since $0\in\Sigma\varsubsetneq S'$,
we have $1 \leq |S' \setminus \Sigma | < \sigma = |S' |$, and we may use the
induction hypothesis which implies that
$$
|G\setminus (A'\wedge^{S'\setminus \Sigma} B')|
\le 2|S'\setminus \Sigma|.
$$
Thus, using Lemma \ref{texnik} (iii) and (ii), we obtain that
\begin{eqnarray*}
|G\setminus (A\wedge^S B) |& = & |G\setminus (A'\wedge^{S'} B') | \\
                        & = & |G\setminus (A'\wedge^{\Sigma} B')|+
|G\setminus (A'\wedge^{(S'\setminus \Sigma)} B')|\\
& \le& 2|\Sigma| +2|S'\setminus \Sigma|\\
& =& 2|S'|\\
&=& 2|S|
\end{eqnarray*}
which finishes the induction step and the proof.
\end{proof}

\noindent {\bf Examples.} {\it
From the characterization given in \cite{Guo09},
we construct the first example. By a slight modification
we can then generate the two other examples,
which show that the bound of Theorem \ref{general} is tight.

\begin{enumerate}
\item Let $G=\mathbb{Z}/15\mathbb{Z}$, we have that $G_0=0$.
If $S=\{0\}$ then $L_S=L(G)=1$. Let us consider $H=\subgp{5}$,
$A=A_0\cup (2+H)$ and $B=B_0\cup (2+H)\cup (4+H)\cup (1+H)$,
where $A_0=B_0=\{0,5\}$. In particular, $|A|+|B|=|G|+L_S$. Then
$$
\nu_{A+B}(u)=\left\{\begin{array}{lc}
                       |A_0|+|H|, & \hbox{if} \ u\in 1+H;\\
                       2|A_0|, & \hbox{if} \ u\in 2+H;\\
                       |H|, & \hbox{if} \ u\in 3+H;\\
                       |A_0|+|H|, & \hbox{if} \ u\in 4+H;\\
                     \end{array}\right. \hspace{0.5cm}
\hbox{and} \ A_0\wedge^{S} B_0=\{5\}.
$$

Thus, the set of $A\wedge^S B$-exceptions is $\{0,10\}$.

\item Let $G=\mathbb{Z}/30\mathbb{Z}$, we have that $G_0=\{0,15\}$.
If $S=\{0,15\}$ then $L_S=L(G)$. Let us consider $H=\subgp{5}$,
$A=A_0\cup (2+H)$ and $B=B_0\cup (2+H)\cup (4+H)\cup (1+H)$,
where $A_0=B_0=\{0,5,15,20\}$. In particular, $|A|+|B|=|G|+L_S$. Then
$\nu_{A+B}(u)>L_S$ for $u\notin H$ and
$$
A_0\wedge^S B_0=\{5,20\}.
$$
Thus, the set of $A\wedge^S B$-exceptions is $\{0,10,15,25\}$.

\item Let $G=\mathbb{Z}/45\mathbb{Z}$,
we have that $G_0=0$. If $S=\{0,15,30\}$ then $L_S=3|G_0|$
(clearly, since $3|G_0|\ge L_S\ge L_S(0)=3$).
Let us consider $H=\subgp{5}$, $A=A_0\cup (2+H)$ and
$B=B_0\cup (2+H)\cup (4+H)\cup (1+H)$,
where $A_0=B_0=\{0,5,15,20,30, 35\}$.
In particular, $|A|+|B|=|G|+L_S$.
Then $\nu_{A+B}(u)>L_S$ for $u\notin H$ and
$$
A_0\wedge^S B_0= \{5,20,35\}.
$$
Thus, the set of $A\wedge^S B$-exceptions is $\{0,10, 15,25,30,40\}$.
\end{enumerate}
}

\section{The critical sets for Abelian groups. Case $L_S=|S|\ L(G)$.}

In what follows, instead of $G_0$ we will write $K(G)$. That is, $K(G)=\{g\in G: \ 2g=0\}.$ We start with a lemma.

\begin{lemma}
\label{lemma_lower_2s}
Let $A,B$ and $S$ be subsets, containing $0$,
of a finite Abelian group $G$. Assume that $|A|+|B|=|G|+L_S$,
that  $0$ is an $A\wedge^{S} B$-exception and that $L_S=|S|\ L(G)$.
Let $z$ and $z'$ be two $A\wedge^{S}B$-exceptions
such that $z \notin S\cup -S$, $z' \in S\Delta -S$ and
$z'-z\notin S\cap -S$. Then,
$2z-2z'\in (S\Delta -S)\cap (A\wedge^{S}B)$.
\end{lemma}

\begin{proof}
Slightly more precisely, we shall in fact prove that
if $\epsilon=-1$ or $1$ and $z'\in \epsilon (S\setminus -S)$,
then $2z-2z'\in \epsilon (S\setminus -S)\cap (A\wedge^{S}B)$.

By Corollary \ref{coroxxl}, we may assume that $z=2b$ and $z'=2b'$
for some $b, b' \in A\cap B$.
Recall first that, by Lemma \ref{PofTheorem7}, we have
\begin{equation}
\label{belle}
z'-z = 2b'-2b\in A\cap B.
\end{equation}

For all this proof, we define $r=2b'-b$.

We first prove that
\begin{equation}
\label{deuxbprime}
z'\in A\cup B.
\end{equation}

Indeed, suppose to the contrary that
$b+r  \notin A\cup B$, therefore since $2b$ is an $A\wedge^{S}B$-exception,
Lemma \ref{clau} implies that $b-r=2b-2b'\in A\cap B$.
But then, the equalities
$$
(2b-2b')+(2b'-2b)=(2b'-2b)+(2b-2b')=0
$$
giving two ways to write $0$ as an element of $A+B$,
by using \eqref{belle}, imply, since $0$ is an $A\wedge^{S}B$-exception,
that $4b-4b'$ and $4b'-4b$ are in $S$, that is $2(2b'-2b)\in S\cap -S$.
Thus, by Lemma \ref{clau},
we obtain $2b'-b= b+(2b'-2b) \in  A\cap B$.
This implies in turn, in view of the writing of the
$A\wedge^{S}B$-exception $2b'$ as an element of
$A+B$ in the form $(2b'-b)+b$, that $2b'-2b \in S$.
In a similar fashion, we obtain $2b-2b'\in S$. Therefore,
$z-z'\in S\cap -S$. Since this is a contradiction with an
assumption of our statement, \eqref{deuxbprime} is proved.

We are now reduced to study two cases.
\smallskip

\noindent \underline{Case 1}: $z' \in A$.

Writing $2b'=2b'+0\in A+B$ and since $2b'$
is an $A\wedge^{S}B$-exception, we obtain that
$2b' \in S$ and thus
$$
z' \in S\setminus -S
$$
(recall that in all this proof, we  assume that
$z'\in S\Delta -S$).
This implies that
$$
b+r=2b'=z'\notin B,
$$
since otherwise it follows from the writing of the
$A\wedge^{S}B$-exception $2b'=0+2b'\in A+B$ that
$2b' \in -S$, a contradiction to $2b' \in S\setminus -S$.

By Lemma \ref{clau},
\begin{equation}
\label{eqeq}
|A\cap \{b-r,b+r\}|+|B\cap \{b-r\}| =
|A\cap \{b-r,b+r\}|+|B\cap \{b-r,b+r\}|\in\{2,3\}.
\end{equation}
Note that, $b-r \neq b+r$ since otherwise we would have
$b-r=b+r \notin B$ and the left-hand side of the preceding
formula would be equal to $1$.
In particular, by Lemma \ref{clau}, we obtain that
\begin{equation}
\label{ratty}
\displaystyle r\notin \mathcal{H}(S\cap -S).
\end{equation}

We now prove that
\begin{equation}
\label{zz'}
z-z' \notin B.
\end{equation}
Indeed if, to the contrary, $2b-2b'\in B$ then the writing
$2b$ as the sum $2b'+(2b-2b')$ in $A+B$ implies,
since $2b$ is an $A\wedge^{S}B$-exception, that $2r = 4b'-2b\in S$.
In view of \eqref{ratty}, we obtain that $2r=s$ for some
$s \in S\setminus -S$. By Lemma \ref{clau}, we derive
$$
|A\cap \{b-r,b+r\}|+|B\cap \{b-r,b+r\}|=3
$$
hence, since $b+r \not\in B$, we obtain that  $2b-2b'=b-r \in A\cap B$.
We are now back to the situation of the proof of assertion \eqref{deuxbprime}.
Proceeding in a similar way, we obtain consecutively that
$2z-2z' \in S\cap -S$, $2b'-b \in A\cap B$, $2b'\in A+B$
and finally $2b'-2b \in S$. With $2b-2b'\in S$, which holds by symmetry,
the contradiction $z-z'\in S\cap -S$ follows and \eqref{zz'} is proved.

Relation \eqref{zz'} can be rewritten as $b-r=2b-2b'\notin B$.
Using $b+r \not\in B$, we see that the left-hand side of \eqref{eqeq}
must be equal to $2$
and we obtain, by Lemma \ref{clau}, that
\begin{equation}
\label{2B2B'}
2b-2b'=b-r \in A\setminus B.
\end{equation}
Thus, using this, \eqref{belle} and the writing $(2b-2b') +(2b'-2b)=0$
yields, since $0$ is an $A\wedge^{S}B$-exception,
$$
4b-4b'\in S.
$$
Hence, we must have
$$
2z-2z' = 4b-4b'\in S\setminus -S,
$$
otherwise $4b-4b'\in S\cap -S$, that is $2(2b-2b')\in S\cap -S$.
Therefore, using that $0$ is an $A\wedge^{S}B$-exception,
Lemma \ref{clau} implies that $2b-2b'\in  A\cap B$,
a contradiction with \eqref{2B2B'}.

Note that, in the present case, if $2z-2z'$ is an
$A\wedge^{S}B$-exception, then by Corollary \ref{coroxxl},
it is of the form $2b''$ for some $b''\in A\cap B$. Thus, using
Lemma \ref{clau}, it follows that $2b-2b' \in b'' + K(G) \subset A \cap B$,
a contradiction.

\smallskip

\noindent \underline{Case 2}: $z' \in B$.
This case is analogous.
\smallskip

Finally, the study of these two cases implies the result.

\end{proof}

The next proposition gives more information
on the structure of $S$ and the set of $A\wedge^{S}B$-exceptions
when $|A\wedge^S B|=|G|-2|S|$.

\begin{proposition}
\label{proof_lower_2s}
Let $A,B$ and $S$ be non-empty subsets of a finite Abelian group $G$ with
$|A|+|B|=|G|+L_S$ and $L_S=|S|\ L(G)$.
If $|A\wedge^S B|=|G|-2|S|$ then
\begin{itemize}
\item[(i)] $S-S\subset 2\cdot G$
\end{itemize}
Moreover, for any $s$ in $S$,
\begin{itemize}
\item[(ii)] we have $-(S-s)=S-s$ and
\item[(iii)] the set of $A\wedge^{S}B$-exceptions can be partitioned
in the form $(z_1+ S)\cup (z_2+S)$, for some $z_1,z_2\in G$.
\end{itemize}
\end{proposition}

\begin{proof}
Let us choose an arbitrary $s$ in $S$.
Since $A+B=G$, the cardinality condition implies that there are
exactly $2|S|$ $A\wedge^SB$-exceptions. Let $w$ be one of them.
By Lemma \ref{xxl} (iv), $w$ can be written as $2b+s$
for some $b \in (A-s) \cap B$.
Let $A'=A-b-s$, $B'=B-b$ and $S'=S-s$. By Lemma \ref{xxl} (iii),
$L_{S}=L_{S'}(0)$ and since we are assuming that $L_{S}=|S|\ L(G)$,
Lemma \ref{L_S} (vi) implies that $S-s=S'\subset 2 \cdot G$. Taking
the union over all possible $s$ implies (i).

In order to prove the equality (ii), notice first that the set
of $A'\wedge^{S'}B'$-exceptions cannot be included in $S'\cup -S'$
since in this case its cardinality would be at most $|S'\cup -S'| < 2
|S'|$, a contradiction.

Let $z$ be an $A'\wedge^{S'}B'$-exception outside $S'\cup -S'$. By
Corollary \ref{CoroPofThm7},
$$
G \setminus (A'\wedge^{S'}B') \subset S'\cup -S' \cup (z+S'\cap -S').
$$
But since the right-hand side has a cardinality at most $|S' \cup -S'|
+ |S'\cap -S'|=2|S'|$ we obtain, using the assumption on the
cardinality of the set of $A'\wedge^{S'}B'$-exceptions, the equality
$$
G \setminus (A'\wedge^{S'}B') = S'\cup -S' \cup (z+S'\cap -S')
$$
and that this union is disjoint. It follows that we can partition
the set of $A'\wedge^{S'}B'$-exceptions as
\begin{equation}
\label{partitionzz}
G \setminus (A'\wedge^{S'}B') =
(S'\cap-S') \cup (S'\Delta -S') \cup (z+S'\cap-S').
\end{equation}

If $S'\Delta -S'$ is empty then $S'=-S'$ and the result is
proved. Otherwise let $z' \in S'\Delta -S'$. By partition
\eqref{partitionzz}, it is not in
$z+S'\cap -S'$ and we may apply Lemma \ref{lemma_lower_2s} which
implies that $2z-2z'$ is an element of $S'\Delta -S'$ not in
the set of $A'\wedge^{S'}B'$-exceptions, contrarily to the partition \eqref{partitionzz}.
This proves (ii).

With (ii), Corollary \ref{CoroPofThm7} and the assumption
$|G\setminus (A'\wedge^{S'}B')|=2|S|$ yield
$$
G \setminus (A'\wedge^{S'}B') = S' \cup (z+ S').
$$
Thus
$$
G \setminus (A\wedge^{S}B) = G \setminus (A'\wedge^{S'}B') + w
= (S' \cup (z+ S'))+w = 
(2b+S) \cup (2b+z+S),
$$
on recalling that $w=2b+s$.
\end{proof}

The next lemma will be useful.

\begin{lemma}
\label{S_general}
Let $G$ be a finite Abelian group. Let $S$ be a subset of $G$
such that, for each $s\in S$, $S-s=-(S-s)$. Then
one of the following happens:
\begin{itemize}
\item[(i)] $S$ is a coset.
\item[(ii)] $S-S$ is contained in $K(G)$.
\item[(iii)] There exists a group $H \leq G$ such that
$S=(s_1+H)\cup (s_2+H)\cup \ldots \cup (s_k+H)$ and
for any $i,j\in \{1,\ldots, k\}$ we have $2(s_i-s_j)\in H$.
\end{itemize}
\end{lemma}

\begin{proof}
First, notice that the assumption is equivalent to the fact that
for any $s \in S$, $S=-S+2s$, in other words
\begin{equation}
\label{Ss}
S=-S+ 2 \cdot S.
\end{equation}

Assume first that $|2 \cdot S|=|S|$. In this case, \eqref{Ss} implies
for any $s \in S$, that $-s+2 \cdot S \subset S$. By the assumption on
the cardinalities of these two sets, we then obtain $2 \cdot S = S+s$
and finally $2\cdot S=2S$. Choose any $s$ in $S$ and write $S'=S-s$ so
that $0$ is in $S'$, we get $S'=-S'$ and $2S' = 2 \cdot S' = S'$.
Therefore, $S'$ is a subgroup of $G$ and $S$ is a coset.

Suppose now that $|2 \cdot S|<|S|$, in particular $K(G)\neq 0$.
The trivial inequality $|S|\le |2 \cdot S|+|S|-1$ implies,
by Theorem \ref{kneser}, that there exists a group $H$ (namely
the stabilizer of $S$) such that
$$
|S|= |-S+ 2 \cdot S|=|2 \cdot S+H|+|-S+H|-|H| =|2 \cdot S+H|+|S+H|-|H|.
$$
In other words, the sum of two non-negative terms
$(|2 \cdot S+H|-|H|)+(|S+H|-|S|)$ is equal to $0$.
Hence, we obtain that $|2 \cdot S+H|=|H|$ and $|S+H|=|S|$, in particular, this means that
$S$ is composed of (full) cosets modulo $H$.

If $H=\{ 0 \}$, then for any $s \in S$, we get $2 \cdot(S-s)=\{ 0 \}$ and
thus $S-s\subset K(G)$ which implies $S-S \subset K(G)$.

Suppose now $H\neq \{ 0 \}$. Let
$S=(s_1+H)\cup (s_2+H)\cup \cdots \cup (s_k+H)$ be the decomposition
of $S$ into $H$-cosets ($H$-tiling). We have
$2\cdot S=(2s_1+H)\cup (2s_2+H)\cup \cdots \cup (2s_k+H)$. Using now
$|2 \cdot S+H|=|H|$, we deduce that
$$
(2s_1+H) = (2s_2+H)= \cdots = (2s_k+H),
$$
that is, $2(s_i-s_j)\in H$ for any pair $i,j\in \{1,\ldots, k\}$.
\end{proof}

The next example shows that a set $S$ that verifies the hypothesis
of Lemma \ref{S_general} it is not necessarily a coset.

\noindent {\bf Example.} {\it
Let $G=\mathbb{Z}/4\mathbb{Z}\oplus \mathbb{Z}/2\mathbb{Z}
$ with generators $a,b$ and relations
$4a=2b=0$. Note that $K(G)=\{0,2a,b,2a+b\}$.
Let us now consider the set
$S=H_0\cup (a+H_0)\cup (a+b+H_0)=H_0\cup (a+K(G))$,
where $H_0=\subgp{2a}=\{ 0,2a \}$. Then we may check that,
for each $s\in S$, we have $-(S-s)=S-s$
(or, equivalently, $S=-S+2\cdot S$) while $S$ is not a coset.
}

We are now ready to study the critical sets $S$ of
Theorem \ref{lower_2s}.

\begin{theorem}
\label{exceptions}
Let $A,B$ and $S$ be non-empty subsets of a finite Abelian
group $G$. Assume that $|A|+|B|=|G|+L_S$, where $L_S=|S|\ L(G)$ and
that $|A\wedge^S B|= |G|-2|S|$. Then
\begin{itemize}
\item[(i)]  there exists a subgroup $H$ of
$2 \cdot G$ such that $S$ is a coset modulo $H$, and
\item[(ii)] $G\setminus (A\wedge^S B )$ is a union of two cosets modulo $H$.
\end{itemize}
\end{theorem}

\begin{proof}
Let $z$ be an $A\wedge^S B$-exception, which exists since $|A\wedge^S B|= |G|-2|S|$.
By Lemma \ref{xxl}, for each $s\in S$ we may find $b \in (A-s)\cap B$ and define
the sets $A'=A-b-s$, $B'=B-b$ and $S'=S-s$. In particular, we have that
$0\in  A'\cap B'\cap S'$ and
$G\setminus (A'\wedge^{S'} B' )=G\setminus (A\wedge^S B )-z$.

Assume that we have proved that $S-s$ is
a subgroup $H$ of $G$. Then $H \subset S-S \subset 2 \cdot G$
by Proposition \ref{proof_lower_2s} (i). Statement (i) of the Theorem
follows. By Proposition \ref{proof_lower_2s} (iii),
we obtain that the set of $A\wedge^{S} B$-exceptions can be partitioned
into two translates of $S$, that is, two $H$-cosets and (ii) is proved.

What remains to be proved is that $S' = S-s$ is a group. Suppose the contrary.
By Proposition \ref{proof_lower_2s} (ii), we can apply Lemma \ref{S_general}.
Two cases may happen.
\smallskip

\noindent \underline{Case 1}. $S -S\subset K(G)$.

Let $s_1',s_2'\in S'$. As shown in the course of the proof of Proposition \ref{proof_lower_2s},
in the present situation: $S' \subset G\setminus (A'\wedge^{S'}B')$.
By Corollary \ref{coroxxl}, there exists
$y\in A'\cap B'$ such that $s_1'=2y$. Notice that since
we have that $S' \subset S-S \subset K(G)$ and $K(G)$
is a group, we have that $s'_1+s'_2 \in K(G)$.

Since $2y\in S'$ is an $A'\wedge^{S'}B'$-exception, and
$0 \in S'\cap -S'$, Lemma \ref{clau} implies
$y+K(G)=y +\mathcal{H}(\{0\}) \subset  A'\cap B'$. Therefore $y+s'_1+s'_2 \in A' \cap B'$.
Writing $s'_2$ as the sum  $s'_2 = (y+s_1'+s_2')+y$ in $A'+B'$,
we deduce from the fact that $s_2'$ is an $A'\wedge^{S'} B'$-exception
that $s_1'+s_2'= (y+s_1'+s_2')-y \in S'$. This proves $S'+S' \subset S'$.

This and the facts that $0\in S'$ and $S'$ is finite implies that
$S'$ is a subgroup of $G$, a contradiction.
\smallskip

\noindent\underline{Case 2}. We now suppose that there exists a group
$J$ such that $S=(s_1+J)\cup (s_2+J)\cup \ldots \cup (s_k+J)$
and for each pair $i,j\in \{1,\ldots, k\}$ we have $2(s_i-s_j)\in J$.

Notice that $J \subset S-s_1 \subset S-S \subset 2\cdot G$ which shows
$J\subset 2\cdot G$.

Let $\pi : G \rightarrow G/J$.
We first prove that
\begin{equation}
\label{quocient}
|\pi (A')|+|\pi (B')|=|G/J|+L_{\pi (S')},\quad
L_{\pi(S')}=|\pi (S')|\ L(G/J)\quad
\text{and} \quad \pi (S')\subset  K(G/J).
\end{equation}

Indeed, since by Proposition \ref{proof_lower_2s} (iii) the set
of $A\wedge^S B$-exceptions can be partitioned into two translates of
$S$ and since $S$ is a union of $J$-cosets,
the equality $|A\wedge^S B|= |G|-2|S|$ implies that
$$
|\pi (A')\wedge^{\pi (S')} \pi (B')|= |G/J|-2|\pi (S')|.
$$
Thus, by Lemma \ref{phS} and Lemma \ref{L_S} (iv), we obtain that:
\begin{equation}
\label{hdr}
|\pi(A')|+|\pi(B')|\le |G/J|+L_{\pi(S')}\le  |G/J|+|\pi(S')||K(G/J)|.
\end{equation}
Let us see now that $|K(G/ J )|=L(G)$. Since $J$ is a subgroup of $2 \cdot G$,
the set $Y=\{y\in G:\ 2y\in J\}$ itself is a group. Moreover, we have
that $Y+K(G)=Y$ and $Y+J=Y$ and we obtain that $|K(G/J )|=|Y/J |=L(G)$.
Hence, we should have $|\pi (A')|+|\pi(B')|=|G/J |+|\pi(S')||K(G/J )|$,
otherwise, using \eqref{hdr}, we derive the inequalities
$$
|G|+|S'|\ L(G)=|J|(|G/J |+|\pi(S')||K(G/J )|)>|J|(|\pi(A')|+|\pi(B')|)\ge
|A'|+|B'|,
$$
which gives us a contradiction with $|A|+|B|=|G|+|S|\ L(G)$.
Clearly, the inclusion $\pi (S')\subset  K(G/J)$ holds.
This proves \eqref{quocient}.

Notice that, from \eqref{quocient} we also conclude that
$\pi(S')-\pi(S')\subset K(G/J)$. Therefore,
by the same reasoning as in Case 1, applied to $\pi (A'), \pi (B')$
and $\pi(S')$, we obtain that $\pi (S')$ is a subgroup in $G/J$.
Which implies, since $S'$ is union of cosets modulo $J$,
that $S'$ is a subgroup of $G$, a contradiction.
\end{proof}

\begin{theorem}
Let $A,B$ and $S$ be non-empty subsets of a finite Abelian group $G$ with
$|A|+|B|=|G|+L_S$ and $L_S=|S|\ L(G)$.
Let $2b_s+s$ be an $A\wedge^S B$-exception, with $s\in S$,
$b_s\in (A-s)\cap B$. Then, the following assertions are equivalent
\begin{itemize}
\item[(i)] $|A\wedge^S B|= |G|-2|S|$.
\item[(ii)]
\begin{itemize}
\item[(a)] the set $\Sigma = S-s$ is a subgroup of $2 \cdot G$,
\item[(b)] if $\pi:G \rightarrow G/ \Sigma$ denotes the natural projection,
$A'=\pi (A-b_s-s)$, $B'=\pi (B-b_s)$,
$$
|A'|+|B'|=|G/ \Sigma |+ L (G/ \Sigma ) \text{ and }
|A'\wedge B'|= |G/ \Sigma |-2.
$$
\end{itemize}
\end{itemize}
\end{theorem}

\begin{proof}
Suppose first that the equality $|A\wedge^S B|= |G|-2|S|$ holds.
By Theorem \ref{exceptions}, $\Sigma$ is a subgroup of $2 \cdot G$ and
the set of $A\wedge^S B$-exceptions is a union of two cosets modulo $\Sigma$.
Hence, since $A\wedge^S B-(2b_s+s)=(A-b_s -s)\wedge^{\Sigma}(B-b_s)$,
we have that $|A'\wedge B'|= |G/\Sigma|-2$.
This implies, by Lemma \ref{phS}, in case $S=\{0\}$,
that $|A'|+|B'|\le|G/ \Sigma |+ L(G/\Sigma )$.

Let us see now that $L(G/ \Sigma )=L(G)$.
Since $\Sigma$ is a subgroup of $2 \cdot G$, the set
$Y=\{y\in G:\ 2y\in \Sigma\}$ itself is a group.
Moreover, we have that $Y+K(G)=Y$ and $Y+\Sigma=Y$. Thus
we obtain that $L (G/\Sigma )=|Y/\Sigma |=L(G)$.
Hence, we should have $|A'|+|B'|=|G/\Sigma |+ L(G/\Sigma )$,
otherwise
$$
|G|+|S|\ L(G)=|S|(|G/\Sigma |+ L (G/\Sigma ))>|S|(|A'|+|B'|)\ge
|A|+|B|,
$$
a contradiction with $|A|+|B|=|G|+|S|\ L(G)$.
Note that, we also conclude that $A$ and $B$ are unions
of cosets modulo $\Sigma$.

For the converse statement, assume that $\Sigma$ is a subgroup of $2 \cdot G$
and that the equalities $|A'|+|B'|=|G/\Sigma |+L(G/\Sigma )$ and
$|A'\wedge B'|= |G/\Sigma|-2$ hold. The condition $|A|+|B|=|G|+L_S$,
with $L_S=|S|\ L(G)$ implies that $A$ and $B$ are unions of cosets
modulo $\Sigma$, since, as it is easy to check,
$L(G/\Sigma)= L(G)$. Hence, the set of $A'\wedge B'$-exceptions is
a union of two cosets modulo $\Sigma$. Therefore, we have
$|A\wedge^S B|= |G|-2|S|$.
\end{proof}

From the previous result together with Theorem \ref{Guo_theo_charac}
(in case $a=0$) we obtain the characterization of the critical sets
$A,B$ and $S$ of Theorem \ref{lower_2s}.

\begin{theorem}
\label{characterization_even}
Let $A,B$ and $S$ be non-empty subsets of a finite Abelian group $G$ with
$|A|+|B|=|G|+L_S$ and $L_S=|S|\ L(G)$. Let $2b_s+s$ be an
$A\wedge^S B$-exception with $s\in S$ and $b_s\in (A-s)\cap B$.
Then, the equality $|A\wedge^S B|= |G|-2|S|$ holds,
if and only if the following conditions hold true:

\begin{itemize}
\item[(i)]  $\Sigma = S-s$ is a subgroup.
\end{itemize}
\noindent Let $\pi: G\rightarrow G/\Sigma$ be the natural projection and
let $A'=\pi (A-b_s-s)$ and $B'=\pi (B-b_s)$.
\begin{itemize}
\item[(ii)] There exists $b\in A'\cap B'$ such that $H=\subgp{2b}$
is a subgroup in $G/\Sigma$ of odd order $d$ greater than $1$.
\item[(iii)] There exists
$\{x_1,\ldots , x_k,x_{k+1},\ldots ,x_l,x_{l+1},\ldots,x_m\}\subset
(G/\Sigma) \setminus (K(G/\Sigma )+H)$, where
$m=(|G/\Sigma |/d- L(G/\Sigma ))/2$ and $0\le k,l\le m$, such that
\begin{eqnarray*}
&& (G/\Sigma) \setminus (K (G/\Sigma )+H)=\
\bigcup_{i=1}^m((x_i+H)\cup (-x_i+H)),\\
&& A'=\{0,b,3b,\ldots, (d-2)b\}+K(G/\Sigma )\cup
(\{x_1,\ldots , x_k,\pm x_{k+1},\ldots ,\pm x_l\}+H),\\
&& B'=\{0,b,3b,\ldots, (d-2)b\}+K(G/\Sigma ) \cup
(\{x_1,\ldots , x_k,\pm x_{l+1},\ldots ,\pm x_m\}+H),
\end{eqnarray*}
\item[(iv)] And $A=b_s+s+\pi^{-1}(A')$ and $B=b_s+\pi^{-1}(B')$.
\end{itemize}
\end{theorem}

Note that, under the assumptions of the previous theorem,
we can determine the set of $A\wedge^SB$-exceptions.
Let $Y=\{y\in G:\ 2y\in \Sigma \}$ then
$\pi^{-1}(K(G/\Sigma ))=Y$ and the set
$G\setminus (A\wedge^SB)=(2b_s+S)\cup (2b^*+2b_s+S)$,
where $b^*\in G$ and $\pi(b^*)=b$.

\section{The critical sets for Abelian groups. Case $L_S<|S|\ L(G)$.}

As an immediate consequence of Theorem \ref{general},
we obtain the following result.

\begin{theorem}\label{tretze}
Let $A,B$ and $S$ be non-empty subsets of a finite Abelian group $G$ with
$|A|+|B|=|G|+L_S$. If the equality $|A\wedge^S B|= |G|-2|S|$ holds
then there exists a decomposition
$S=S_1\cup S_2\cup\ldots \cup S_k$ of $S$ modulo $2 \cdot G$ such that
\begin{enumerate}
\item[(i)] $|S_i|=L_S/L(G)$ for each $i\in \{1,2,\ldots, k\}$,
\item[(ii)] $|A\wedge^{S_i}B|=|G|-2|S_i|$ and
$|A\wedge^{(S\setminus S_i)}B|=|G|-2|S\setminus S_i|$ for each
$i\in \{1,2,\ldots, k\}$. Moreover, the set of $A\wedge^SB$-exceptions
can be partitioned in the form:
$$
G\setminus (A\wedge^SB)=\cup_{i=1}^k(G\setminus (A\wedge^{S_i}B)).
$$
\end{enumerate}
\end{theorem}

\begin{proof}
Let $L_S=mL(G)$, for some $m\le |S|$. We will prove the result by induction on $|S|$.
We have $A+B=G$. Let $z$ be an $A\wedge^S B$-exception.
We write $z=2b+s$ for some $s\in S$ and
$b\in (A-s)\cap B$ and define $A'=A-b-s$, $B'=B-b$ and $S'=S-s$.
Again, $0\in A'\cap B'\cap S'$, $G\setminus (A\wedge^SB)-(2b+s)=G\setminus(A'\wedge^{S'}B')$ (in particular, $0$ is an $A'\wedge^{S'}
B'$-exception), $L_S=L_{S'}(0)=L_{S'}$ and $A' + B' =G$.

We denote $\Sigma = S' \cap 2 \cdot G$. Notice that $0\in \Sigma$,
and also by Lemma \ref{texnik} (i) that $m=|\Sigma|$.
If $|S'|=m$ then the result holds with $k=1$. Suppose now that $|S'|>m$.

By Lemma \ref{texnik} (iii), we obtain that the set of $A'\wedge^{S'}B'$-exceptions can be partitioned as follows
\begin{eqnarray}\label{final}
G\setminus (A'\wedge^{S'} B') & = &  (G\setminus (A'\wedge^{\Sigma} B'))\cup
(G\setminus (A'\wedge^{(S'\setminus \Sigma)} B'))
\end{eqnarray}

Now, using Lemma \ref{texnik} (ii), we have that $|G\setminus (A'\wedge^{\Sigma} B')|\le 2|\Sigma|$. In particular, since $|S'|>|\Sigma|$, there exists and exception $z'\in G\setminus (A'\wedge^{(S'\setminus \Sigma)} B')$. Thus, \eqref{lesl} is still valid
and we have $L_{S' \setminus \Sigma} = L_S$, hence
$|A'|+|B'|=|G|+L_{S' \setminus \Sigma}$.  By Theorem \ref{general}, we obtain that $|G\setminus (A'\wedge^{(S'\setminus \Sigma)} B')|\le 2|S'\setminus \Sigma|$. Therefore, using $|G\setminus (A\wedge^S B)|=2|S|$ and the partition (\ref{final}), we conclude that $|G\setminus (A'\wedge^{\Sigma} B')|= 2|\Sigma|$ and $|G\setminus (A'\wedge^{(S'\setminus \Sigma)} B')|=2|S'\setminus \Sigma|$.

The inductive process applied to $S'\setminus \Sigma$ completes the result.

\end{proof}

We now introduce some results that allow us to be more precise about
the structure of the sets $S$ and $G\setminus (A\wedge^SB)$ in the case
$L_S<|S|\ L(G)$, provided some restriction holds.

\begin{proposition}
\label{catorze}
Let $A$ and $B$ be subsets of a finite Abelian group $G$.
Suppose that there exists an element $b\in A\cap B$ such that:
\begin{itemize}
\item[(i)] The order $d$ of the subgroup $H=\subgp{2b}$ is
an odd integer greater than $2 L(G)-1$.
\item[(ii)] There exists a subset
$\{x_1,\ldots ,x_k,x_{k+1},\ldots, x_l,x_{l+1},\ldots, x_m\}
\subset G \setminus (K(G)+H)$, where
$m=(|G|/d-|K(G)|)/2$ and $0\le k,l\le m$, such that
\begin{eqnarray*}
&& G\setminus (K(G)+H)=\bigcup_{i=1}^m((x_i+H)\cup (-x_i+H)),\\
&& A=(\{0,b,3b,\ldots, (d-2)b\}+K(G))\cup
(\{x_1,\ldots,x_k,\pm x_{k+1},\ldots, \pm x_l\}+H), \\
&& B=(\{0,b,3b,\ldots, (d-2)b)\}+K(G))\cup
(\{x_1,\ldots,x_k,\pm x_{l+1},\ldots, \pm x_m\}+H).
\end{eqnarray*}
\end{itemize}
Then, for each $z\in G\setminus (K(G)\cup (2b+K(G)))$ we have
$\nu (z)>L(G)$. In particular, if $\nu (z)=L(G)$ then
$z\in K(G)\cup (2b+K(G)).$
\end{proposition}

\begin{proof}
Let $P,Q$ and $R$ be the subsets defined as follows:
\begin{eqnarray*}
P &=& \{0,b,3b,\ldots, (d-2)b\}+K(G), \\
  Q &=& \{x_1,\ldots,x_k,\pm x_{k+1},\ldots, \pm x_l\}+H, \\
  R &=& \{x_1,\ldots,x_k,\pm x_{l+1},\ldots, \pm x_m\}+H.
\end{eqnarray*}

Then $A=P\cup Q$, $B=P\cup R$ and
$A+B=(P+P)\cup (P+R)\cup (P+Q)\cup (Q+R)$. Notice that,
since $|H|=d$ is odd, we can prove that $|H+K(G)|=d L(G)$
and thus $|A|+|B|=|G|+L(G)$ (see \cite{Guo09}). In particular, $A+B=G$.
Let us see now how $\nu (z)$ is lower bounded according
to which part of the sum it belongs to.

For the sum $P+P$, clearly, $P+P\subset K(G)+H$ and since $d$
is odd, we have that  $db+K(G)=K(G)$. Thus, by a direct computation,
it follows that $\nu (z)>L(G)$ unless $z\in K(G)\cup (2b+K(G))$.

Concerning the sum $ P+R$, we have that $P+R\subset G\setminus (K(G)+H)$.
Since $\{0,b,3b,\ldots, (d-2)b\}\subset H$, it is clear that $P+R$
is a union of $(H+K(G))$-cosets. Moreover, for each $z\in P+R$
 we have $\nu (z)\ge (d+1)/2$. Hence, since we are assuming that
$(d+1)/2>L(G)$, we conclude that $\nu (z)> L(G)$, for each $z\in P+R$.

The case of the sum $P+Q$ is similar to the previous one.
(Here we have that $P+Q\subset G\setminus (K(G)+H)$).

Finally, it follows that $Q+R$ is a union of $H$ cosets, none of which
is $H$, hence $0, 2b\notin Q+R$. Moreover, for each $z\in Q+R$,
$\nu (z)\ge |H|> L(G)$.
\end{proof}

With the notation and assumptions of the preceding proposition,
we can state the following two corollaries:

\begin{corollary}
\label{quinze}
Let $z\in G$. If $\nu (z)=L(G)$ then
$z\in (K(G)\cup (2b+K(G)))\setminus (Q+R).$
\end{corollary}

\begin{corollary}
\label{setze}
Let $A,B$ be non-empty subsets of a finite Abelian group $G$ and $T$
be a subgroup of $2 \cdot G$ such that $|A|+|B|=|G|+|T|\ L(G)$.
Let $\pi: G\rightarrow G/T$ be the natural projection.
Denote $A'=\pi (A)$ and $B'=\pi (B)$.
Suppose that
\begin{itemize}
\item[(i)] there exists $b\in A'\cap B'$ such that $H=\subgp{2b}$
is a subgroup in $G/T$ of odd order $d$ greater than $2L(G/T)-1$,
\item[(ii)] there exists
$\{x_1,\ldots , x_k,x_{k+1},\ldots ,x_l,x_{l+1},\ldots,x_m\}
\subset (G/T)\setminus (K(G/T)+H)$, where
$m=(|G/T|/d- L(G/T))/2$ and $0\le k,l\le m$,
such that
\begin{eqnarray*}
&& (G/T) \setminus (K(G/T)+H)=
\bigcup_{i=1}^m((x_i+H)\cup (-x_i+H)),\\
&& A'=\{0,b,3b,\ldots, (d-2)b\}+K (G/T) \cup
(\{x_1,\ldots , x_k,\pm x_{k+1},\ldots ,\pm x_l\}+H),\\
&& B'=\{0,b,3b,\ldots, (d-2)b\}+K (G/T) \cup
(\{x_1,\ldots , x_k,\pm x_{l+1},\ldots ,\pm x_m\}+H).
\end{eqnarray*}
\end{itemize}
Then, for each
$z\in G\setminus \pi^{-1}(K(G/T)\cup (2b+K(G/T)))$
we have $\nu (z)>L(G)\ |T|$.
In particular, if $\nu (z)=L(G)\ |T|$ then
$z\in \pi^{-1}(K (G/T)\cup (2b+K(G/T))).$
\end{corollary}

The next theorem gives a characterization of the critical sets
of Theorem \ref{general}, provided a restriction holds.

\begin{theorem}
\label{characterization_general}
Let $A,B$ and $S$ be non-empty subsets of a finite Abelian group $G$ with
$|A|+|B|=|G|+L_S$. Assume that $0$ is an $A\wedge^S B$-exception
with $0\in A\cap B\cap S$ and that $|A\wedge^S B|= |G|-2|S|$.
The following conditions hold true:
\begin{itemize}
\item[(i)] There exists a decomposition
$S=S_1\cup S_2\cup\ldots \cup S_k$ of $S$ modulo $2 \cdot G$ such that
$|S_i|=L_S/L(G)$ for each $i\in \{1,2,\ldots, k\}$ and
$|A\wedge^{S_i}B|=|G|-2|S_i|$ for each $i\in \{1,2,\ldots, k\}$.
Moreover, the set of $A\wedge^SB$-exceptions can be partitioned
in the form
$G\setminus (A\wedge^SB)=\cup_{i=1}^k(G\setminus (A\wedge^{S_i}B)).$
\end{itemize}
Without loss of generality, we can assume that $0\in S_1$.
\begin{itemize}
\item[(ii)]  $S_1$ is a group.
\end{itemize}
\noindent Let $\pi: G\rightarrow G/S_1$ be the natural projection and
let $A'=\pi (A)$ and $B'=\pi (B)$.
\begin{itemize}
\item[(iii)] There exists $b\in A'\cap B'$ such that
$H=\subgp{2b}$ is a subgroup of $G/S_1$ of odd order $d$ greater than $1$.
\item[(iv)] There exists
$\{x_1,\ldots , x_k,x_{k+1},\ldots ,x_l,x_{l+1},\ldots,x_m\}
\subset (G/S_1)\setminus (K(G/S_1)+H)$, where $m=(|G/S_1|/d-L(G/S_1))/2$
and $0\le k,l\le m$, such that
\begin{eqnarray*}
&& G/S_1\setminus (K(G/S_1)+H)=
\bigcup_{i=1}^m((x_i+H)\cup (-x_i+H)),\\
&& A'=\{0,b,3b,\ldots, (d-2)b\}+K(G/S_1)) \cup
(\{x_1,\ldots , x_k,\pm x_{k+1},\ldots ,\pm x_l\}+H),\\
&& B'=\{0,b,3b,\ldots, (d-2)b\}+K(G/S_1)) \cup
(\{x_1,\ldots , x_k,\pm x_{l+1},\ldots ,\pm x_m\}+H).
\end{eqnarray*}
\end{itemize}

\noindent Moreover, if we suppose that $(d+1)/2> L (G/S_1)$ then
\begin{itemize}
\item[(v)] $S_i=y_i+S_1$, for each $i=1,\ldots, k$, where
$Y=\{y\in G: 2y \in S_1\}$ and $Y=\cup_{i=1}^r (y_i+S_1)$, $r\ge k$,
is a decomposition modulo $S_1$. In particular, $2 \cdot S\subset S_1$,
\item[(vi)] The set of $A\wedge^SB$-exceptions is $\{0,2b^*\}+S$,
where $b^*\in \pi^{-1}(b)$, and
\item[(vii)]
$(\{0,2b\}+\{\pi (y_1),\pi (y_2),\ldots ,\pi (y_k)\})\cap (Q+R)=\emptyset$,
where $Q= \{x_1,\ldots,x_k,\pm x_{k+1},\ldots, \pm x_l\}+H$ and
$R=\{x_1,\ldots,x_k,\pm x_{l+1},\ldots, \pm x_m\}+H$.
\end{itemize}
\end{theorem}

\begin{proof} (i) holds by Theorem \ref{tretze}.
Theorem \ref{characterization_even} implies (ii), (iii) and (iv).
Finally, by Proposition \ref{catorze},  Corollary \ref{quinze} and
Corollary \ref{setze} we obtain (v), (vi) and (vii).
\end{proof}

\section*{Acknowledgements}

Research done when the second author was visiting Universit\'e
Pierre et Marie Curie, E. Combinatoire, Paris,
supported by the Ministry of Education, Spain,
under the National Mobility Programme of Human Resources,
Spanish National Programme I-D-I 2008--2011.

\end{document}